\documentclass[3p]{elsarticle}

\usepackage{hyperref}
\usepackage{graphicx,url}
\usepackage[english]{babel}
\usepackage[utf8]{inputenc}
\usepackage{amsmath}
\usepackage{amssymb}
\usepackage{amsthm}

\newtheorem{theorem}{Theorem}
\newtheorem{corollary}{Corollary}

\newtheorem{lemma}{Lemma}

\newtheorem*{conjecture-linial}{Linial's Conjecture~\cite{Linial1981}}
\newtheorem*{conjecture-dual-linial}{Linial's Dual Conjecture~\cite{Linial1981}}

\usepackage{xspace} 

\newcommand{\Set}[1]{\{#1\}}

\newcommand{\Min}[1]{\min \{#1\}}

\newcommand{\concat}{\circ}

\newcommand{\sB}{\mathcal{B}}
\newcommand{\sC}{\mathcal{C}}

\newcommand{\sP}{\mathcal{P}}
\newcommand{\sQ}{\mathcal{Q}}

\usepackage{thmtools}

\begin{document}

\begin{frontmatter}

\title{On Linial's Conjecture for Spine Digraphs}

\author[unicamp]{Maycon Sambinelli\corref{cor}}
\ead{msambinelli@ic.unicamp.br}
\cortext[cor]{Corresponding author}

\author[ufscar]{Cândida Nunes da Silva}
\ead{candida@ufscar.br}

\author[unicamp]{Orlando Lee}
\ead{lee@ic.unicamp.br}

\address[unicamp]{Institute of Computing, University of Campinas, Campinas, São Paulo, Brazil}
\address[ufscar]{Department of Computing, Federal University of São Carlos, Sorocaba, São Paulo, Brazil}

\begin{abstract}
In this paper we introduce a superclass of split digraphs, which we call  \textbf{spine digraphs}. Those are the digraphs $D$ whose vertex set can be partitioned into two sets $X$
and $Y$ such that the subdigraph induced by $X$ is traceable and $Y$ is a stable set. We also show that Linial's Conjecture holds for spine digraphs.

\end{abstract}


\end{frontmatter}


\section{Introduction}

The digraphs  considered in this text do not contain loops or parallel arcs (but may contain cycles of length two).
Let $D$  be a  digraph. We denote the set of vertices of $D$ by $V(D)$ and the set of arcs of $D$ by $A(D)$.
We use $(u,v)$ to denote an arc with \textbf{head} $v$ and \textbf{tail} $u$.
We say that $u$ and $v$ are \textbf{adjacent} if $(u, v) \in A(D)$ or $(v, u) \in A(D)$.
By a path of $D$, we mean a directed path of $D$ and by a stable set of $D$, we mean a stable set of the underlying graph of $D$.
We denote by $V(P)$ the set of vertices of a path $P$ and the \textbf{size} of a path $P$, denoted by $|P|$, is $|V(P)|$\footnote{Usually $|P|$ denotes the length of a path (number of arcs), but here it denotes the number of vertices.}.
We denote by $\lambda(D)$ the size of the longest path in $D$ and by $\alpha(D)$ the size of a maximum stable set.
A \textbf{path partition} of $D$ is a set of vertex-disjoint paths of $D$ that cover $V(D)$.
We say  that $\sP$ is an \textbf{optimal} path partition of $D$ if there is no path partition $\sP'$ of $D$ such that $|\sP'| < |\sP|$.
We denote by $\pi(D)$ the size of an optimal path partition of a digraph $D$.

Dilworth~\cite{Dilworth1950} showed that for every transitive acyclic digraph $D$ we have $\pi(D) = \alpha(D)$.
Note that this equality is not valid for every digraph; for example, if $D$ is a directed cycle with 5 vertices, then $\pi(D) = 1$ and $\alpha(D) = 2$.
However, Gallai and Milgram~\cite{GallaiMilgram1960} have shown that $\pi(D) \leq \alpha(D)$ for every digraph $D$.

Greene and Kleitman~\cite{GreeneKleitman1976} proved a generalization of Dilworth's Theorem, which we describe next.
Let $k$ be a positive integer.
The \textbf{$\boldsymbol{k}$-norm} of a path partition $\sP$, denoted by $|\sP|_k$, is defined as $|\sP|_k = \sum_{P \in \sP} \Min{|P|, k}$.
We say that $\sP$ is a \textbf{$\boldsymbol{k}$-optimal path partition} of $D$ if there is no path partition $\sP'$ such that $|\sP'|_k < |\sP|_k$.
We denote by $\pi_k(D)$ the $k$-norm of a $k$-optimal path partition of  $D$.
A \textbf{$\boldsymbol{k}$-partial coloring} $\sC^k$ is a set of $k$ disjoint stable sets called \textbf{color classes} (empty color classes are allowed).
The \textbf{weight}  of a $k$-partial coloring $\sC^k$, denoted by $||\sC^k||$, is defined as $||\sC^k|| = \sum_{C \in \sC^k} |C|$.
We say that $\sC^k$ is an \textbf{optimal $\boldsymbol{k}$-partial coloring} of $D$ if there is no $k$-partial coloring $\sB^k$ such that $||\sB^k|| > ||\sC^k||$.
We denote by $\alpha_k(D)$ the weight of an optimal $k$-partial coloring of  $D$.
Given these definitions, what Greene and Kleitman~\cite{GreeneKleitman1976} showed was that for every transitive acyclic digraph $D$, we have $\pi_k(D) = \alpha_k(D)$.
Note that $\pi(D) = \pi_1(D)$ and $\alpha(D) = \alpha_1(D)$.
Thus, Dilworth's Theorem is a particular case of Greene-Kleitman's Theorem in which $k = 1$.

As Gallai-Milgram's Theorem extends Dilworth's Theorem, it is a natural question whether Greene-Kleitman's Theorem can be extended to digraphs in general.
More precisely, is it true that for every digraph $D$ we have that $\pi_k(D) \leq \alpha_k(D)$?
Linial~\cite{Linial1981} conjectured that the answer for this question is positive.

\bigskip
\begin{conjecture-linial}
  \label{co:linial}
  Let $D$ be a digraph and $k$ be a positive integer. Then, $\pi_k(D) \leq \alpha_k(D)$.
\end{conjecture-linial}

Linial's Conjecture remains open, but we know it holds for  acyclic digraphs~\cite{Linial1981}, bipartite digraphs~\cite{Berge1982}, digraphs which contain a Hamiltonian path~\cite{Berge1982}, $k = 1$~\cite{Linial1978}, $k = 2$~\cite{BergerHartman2008} and $k \geq \lambda(D) - 3$~\cite{Herskovics2013}.
For more about this problem, we refer you to the survey presented by Hartman~\cite{Hartman2006}.

Linial also introduced a somewhat dual problem, which we are going to call as \textbf{Linial's Dual Conjecture}, in which the roles of paths and stable sets are exchanged.
To properly state that, we need a few definitions first.
Let $D$ be a digraph and $k$ a positive integer.
A \textbf{$\mathbf{k}$-path} in $D$ is a set of $k$ disjoint paths of $D$ (we allow empty paths).
The \textbf{weight} of a $k$-path $\sP^k$, denoted by $||\sP^k||$, is defined as $||\sP^k|| = \sum_{P\in\sP^k}|P|$.
We say that $\sP^k$ is an \textbf{optimal $\mathbf{k}$-path} of $D$ if there is no $k$-path $\sQ^k$ of $D$ such that $||\sQ^k|| > ||\sP^k||$.
We denote by $\lambda_k(D)$ the weight of an optimal $k$-path of $D$.
A \textbf{coloring} of $D$ is a partition of $V(D)$ into stable sets.
The \textbf{$\mathbf{k}$-norm} of a coloring $\sC=\{C_1,\ldots,C_t\}$, denoted by $|\sC|_k$, is defined as $|\sC|_k=\sum_{C\in\sC}\Min{|C|,k}$.
We say that $\sC$ is a \textbf{$\mathbf{k}$-optimal coloring} of $D$ if there is no coloring $\sC'$ of $D$ such that $|\sC'|_k < |\sC|_k$.
We denote by $\chi_k(D)$ the $k$-norm of a $k$-optimal coloring of $D$.

\begin{conjecture-dual-linial}
  \label{th:Linial-dual}
  Let $D$ be a digraph and $k$ be a positive integer. Then, $\chi_k(D) \leq \lambda_k(D)$.
\end{conjecture-dual-linial}

This conjecture also remains open and, like Linial's Conjecture, we know it holds for some particular cases, such as acyclic digraphs~\cite{AharoniEtAl1985}, bipartite digraphs~\cite{HartmanEtAl1994}, $k = 1$~\cite{Gallai1968,Roy1967}, $k \geq \pi(D)$ (trivial, since $\lambda_k(D) = |V(D)|$), and split digraphs~\cite{HartmanEtAl1994}, which we define next.

Recall that our digraphs may have no loops nor parallel arcs.
A \textbf{semi-complete digraph} is a digraph $D$ such that for every pair of distinct vertices $u,v$, $(u,v)\in A(D)$ or $(v,u)\in A(D)$ or both.
 A \textbf{tournament} is a digraph $D$ such that for every pair of distinct vertices $u,v$, either $(u,v)\in A(D)$ or $(v,u)\in A(D)$.
 Rédei~\cite{Redei1934} proved that every tournament (and hence, every semi-complete digraph) is \textbf{traceable} (i.~e. contains a Hamiltonian path).

For a digraph $D$ and $X\subseteq V(D)$, we denote by $D[X]$ the subdigraph of $D$ induced by $X$.
A digraph $D$ is a \textbf{split digraph} if there exists a partition $\{X,Y\}$ of $D$ such that $D[X]$ is a semi-complete digraph and $Y$ is a stable set of $D$.

Hartman, Saleh and Hershkowitz~\cite{HartmanEtAl1994} proved that $\chi_k(D) \leq \lambda_k(D)$ (Linial's Dual Conjecture) for every split digraph.
In fact, their proof can be extended to a superclass of split digraphs which we introduce next. We say that $D$ is a \textbf{spine digraph} if there exists a partition $\{X,Y\}$ of $V(D)$ such that $D[X]$ is traceable and $Y$ is a stable set in $D$. In this paper we prove Linial's Conjecture for spine digraphs.
We shall use the notation $D[X,Y]$ to indicate that $D$ is a spine digraph with such partition $\{X,Y\}$.

\section{Linial's conjecture for spine digraphs}
\label{sec:split}

First let us discuss the general idea of the proof of Hartman, Saleh and Hershkowitz~\cite{HartmanEtAl1994}  that $\chi_k(D) \leq \lambda_k(D)$ for every spine digraph $D[X,Y]$.
They first showed that $\chi_k(D) \leq |X| + k$ and $\lambda_k(D) \geq |X| + k - 1$ by exhibiting appropriate coloring and $k$-path.
If $\chi_k(D) \leq |X| + k - 1$, then the result follows.
Therefore, the critical case is when $\chi_k(D) = |X| + k$.
In this case, they showed that $\lambda_k(D) \geq |X| + k$ by constructing a $k$-path with such weight.

We follow the same strategy.
However, here the critical case (described later) is more complicated.
We begin by presenting simple bounds for $\pi_k(D)$ and $\alpha_k(D)$.

\bigskip
\begin{lemma}
\label{lem:upper_bound_linial}
 Let $D[X,Y]$ be a spine digraph.
 Then, $\pi_k(D) \leq |Y| + \Min{|X|, k}$.
\end{lemma}
\begin{proof}
  Let $P$ be a Hamiltonian path in $D[X]$ and $\sP = \Set{P} \cup \Set{(y) : y \in Y}$.
  Clearly, $\sP$ is a path partition of $D$ for which $|\sP|_k = \Min{|X|, k} + |Y|$.
  Therefore, $\pi_k(D) \leq |\sP|_k = \Min{|X|, k} + |Y|$.
\end{proof}

\begin{lemma}
\label{lem:lower_bound_linial}
  Let $D[X,Y]$ be a spine digraph.
  Then, $\alpha_k(D) \geq |Y| + \Min{|X|, k - 1}$.
  Moreover, if $|X| < k$, then $\alpha_k(D) = |V(D)|$.
\end{lemma}
\begin{proof}
  First, suppose that $|X| < k$.
  Let $\sC^k = \Set{Y} \cup \Set{\Set{x}: x \in X}$.
  Note that $\sC^k$ is a  $k$-partial coloring of $D$ with $||\sC^k|| = |V(D)|$.
  Therefore, $\alpha_k(D) = ||\sC^k|| = |Y| + |X| = |Y| + \Min{|X|, k - 1}$ and the result follows.
  Thus assume that $|X| \geq k$.
  Take $S\subset X$ such that $|S| = k - 1$, and let $\sC^k = \Set{Y} \cup \Set{\{x\} : x \in S}$.
  Clearly, $\sC^k$  is a $k$-partial coloring for which $||\sC^k|| = |Y| + k - 1$.
  Therefore, $\alpha_k(D) \geq ||\sC^k|| = |Y| + k - 1 = |Y| + \Min{|X|, k - 1}$.
\end{proof}

A spine digraph $D[X,Y]$ is \textbf{$\boldsymbol{k}$-loose} if either $|X| < k$ or there is a set $S \subseteq X$ such that $|S| = k$  and no vertex $y \in Y$ is adjacent to every vertex in $S$.
A spine digraph $D[X,Y]$ that is not $k$-loose is called \textbf{$\boldsymbol{k}$-tight}.

\begin{lemma}
\label{lem:cara_loose}
  If $D[X,Y]$ is a $k$-loose spine digraph, then $\alpha_k(D) \geq |Y| + \Min{|X|, k}$.
\end{lemma}
\begin{proof}
  If $|X| < k$, then by Lemma~\ref{lem:lower_bound_linial}, $\alpha_k(D) = |V(D)| = |Y| + |X| = |Y| + \Min{|X|, k}$.
  We may thus assume that $|X| \geq k$.
  So, there exists $S \subseteq X$ such that $|S| = k$  and no vertex $y \in Y$ is adjacent to every vertex in $S$.
  Suppose that $S = \{x_1, x_2, \ldots, x_k\}$ and let $\sC^k_0 = \{C_1, C_2, \ldots, C_k\}$ be a $k$-partial coloring in which $C_i = \{x_i\}$ for $i = 1, 2, \ldots, k$.
  For each $y \in Y$, choose some vertex $x_i$ not adjacent to $y$ (which exists by the choice of $S$) and add $y$ in color class $C_i$.
  The $k$-partial coloring $\sC^k$ thus obtained has  weight $|Y| + k = |Y| + \Min{|X|, k}$.
  Therefore, $\alpha_k(D) \geq ||\sC^k|| = |Y| + \Min{|X|, k}$.
\end{proof}

\begin{theorem}
\label{the:k_loose}
   Let $D[X,Y]$ be a $k$-loose spine digraph.
   Then, $\pi_k(D) \leq \alpha_k(D)$.
\end{theorem}
\begin{proof}
  By Lemma~\ref{lem:cara_loose}, $\alpha_k(D) \geq |Y| + \Min{|X|, k}$.
  On the other hand, by Lemma~\ref{lem:upper_bound_linial}, $\pi_k(D) \leq |Y| + \Min{|X|, k}$ and the result follows.
\end{proof}

\begin{lemma}
\label{lem:path_cardinaliti_x_plus_one}
   Let $D[X,Y]$ be a spine digraph such that $\lambda(D) > |X|$.
   Then, $\pi_k(D) \leq \alpha_k(D)$.
\end{lemma}
\begin{proof}
  If $\alpha_k(D) = |V(D)|$, then the result follows trivially.
  Thus, we may assume that $\alpha_k(D) < |V(D)|$.
  By Lemma~\ref{lem:lower_bound_linial}, we have that $|X| \geq k$ and also that $\alpha_k(D) \geq |Y| + \Min{|X|, k - 1} = |Y| + k - 1$.
  Since $\lambda(D) > |X|$, there exists a path $P$ in $D$ such that $|P| = |X| + 1$.
  Let $\sP = \Set{P} \cup \Set{(v) : v \notin V(P)}$.
  Clearly, $\sP$ is a path partition of $D$ and $|\sP|_k = \Min{|P|, k} + |Y| - 1 = |Y| + k - 1$.
  Therefore, $\pi_k(D) \leq |\sP|_k = |Y| + k - 1  \leq \alpha_k(D)$.
\end{proof}

In view of the two preceding results, in order to complete the proof of Linial's Conjecture for spine digraphs, we must deal with the case in which $D$ is $k$-tight and $\lambda(D) \leq |X|$. To do so, we present two auxiliary lemmas; but first, we need some definitions.

Given a path $P= (x_1, x_2, \ldots, x_\ell)$, we denote by $ter(P)$ the terminal vertex of $P$, namely $x_\ell$. The subpath $(x_1, x_2, \ldots, x_i)$ is denoted by $Px_i$ and the subpath $(x_i, x_{i+1}, \ldots, x_\ell)$ is denoted by $x_iP$. We denote by $W \concat Q$ the concatenation of two paths $W$ and $Q$.

Let $D[X,Y]$ be a spine digraph and let $P = (x_1, x_2, \ldots, x_\ell)$ be a Hamiltonian path of $D[X]$.
We say that the Hamiltonian path $P$ is \textbf{zigzag-free} in $D$ if there is no vertex $y \in Y$ such that $(y, x_1) \in A(D)$, or $(x_\ell, y) \in A(D)$, or $(x_i, y) \in A(D)$ and $(y, x_{i + 1}) \in A(D)$.




\begin{lemma}
  \label{lem:maria-vai-com-as-outras}
 Let $D[X, Y]$ be a spine digraph, let $P = (x_1, x_2, \ldots, x_\ell)$ be a Hamiltonian zigzag-free path of $D[X]$ and let $y \in Y$ be a vertex adjacent to the first $t$ vertices of $P$.
  Then $(x_i, y) \in A(D)$ for $i = 1, 2, \ldots, t$.
\end{lemma}
\begin{proof}
  The proof is by induction on $t$.
  If $t = 1$, then the result is obvious.
  Now, suppose that $t > 1$.
  By induction hypothesis, we have that $(x_i, y) \in A(D)$ for $i = 1, 2 \ldots, t - 1$.
  If $(y, x_t) \in A(D)$, then $P$ is not zigzag-free in $D$.
  Hence, $(x_t, y) \in A(D)$ and the result follows.
\end{proof}

\begin{lemma}
\label{lem:fish-bone-paths}
  Let $D[X, Y]$ be a $k$-tight spine digraph and let $P = (x_1, x_2, \ldots, x_\ell)$ be a Hamiltonian zigzag-free path of $D[X]$.
  Then, there exist paths $P_1$ and $P_2$ such that:

  \begin{enumerate}[(i)]
    \item $V(P_1) \cap V(P_2) = \varnothing$;
    \item $|P_1| + |P_2| = |X| + k + 1$;
    \item $ter(P_1) \cup ter(P_2) = \Set{x_\ell, y}$, for some $y \in Y$;
    \item $X \subseteq V(P_1) \cup V(P_2)$.
  \end{enumerate}
\end{lemma}
\begin{proof}
  The proof is by induction on $k$.
  Suppose that $k = 1$.
  Since $D$ is 1-tight, we know that every $x_i \in X$ is adjacent to at least one vertex in $Y$.
  Let $y' \in Y$ be a vertex adjacent to $x_1$.
  Since $P$ is zigzag-free in $D$, we have that $(x_1, y') \in A(D)$.
  Among all arcs $(x_i, y) \in A(D)$ with $y \in Y$ and $1 \le i \le \ell$, choose an arc $a$ such that $i$ is maximum.
  Since $(x_1, y') \in A(D)$, one such arc exists.
  As $P$ is zigzag-free in $D$, we have that $i < \ell$ and so the vertex $x_{i + 1}$  exists.
  Let $y'' \in Y$ be a vertex adjacent to $x_{i + 1}$.
  By the choice of $a$, we have that $(y'', x_{i + 1}) \in A(D)$.
  Since $P$ is zigzag-free in $D$, we conclude that $y''\neq y$.
  Therefore, we have that $P_1 = Px_i \concat (x_i, y)$ and $P_2 = (y'', x_{i + 1}) \concat x_{i + 1}P$  meet the conditions (i) through (iv) above.
  This concludes the base case.


  Now, suppose that $k > 1$.
  Since $D$ is $k$-tight, then $|X| \geq k$ and there exists a vertex $y^* \in Y$ which is adjacent to every vertex of $S = \Set{x_1, x_2, \ldots, x_k}$, the set of the $k$ first vertices of $P$.
  By Lemma~\ref{lem:maria-vai-com-as-outras}, we have that $(x_i, y^*) \in A(D)$ for every vertex $x_i \in S$.
  In particular, $(x_k, y^*) \in A(D)$.
  Among all arcs $(x_i, y) \in A(D)$ with $y \in Y$ and $1 \le i \le \ell$, choose an arc $a$ such that $i$ is maximum.
  Note that such arc $a$ exists and that $i \geq k$, since $(x_k, y^*) \in A(D)$.
  As $P$ is zigzag-free in $D$, we have that $i < \ell$ and so the vertex $x_{i + 1}$  exists.
  Note that by the choice of $i$, if some vertex $y' \in Y$ is adjacent to $x_{i + 1}$ then $(y', x_{i + 1}) \in A(D)$.

  Let $X' = V(Px_{i})$ and let
  \[Y' = \Set{y' : y' \in Y \textrm{ and $y'$ is adjacent to $x_{i + 1}$}}.\]
  Let $D' = D[X' \cup Y']$.
  Clearly, $D'$ is a spine digraph.
  Let $P' = Px_{i}$.
  To show that $P'$ is zigzag-free in $D'$, suppose the contrary.
  Since $P$ is zigzag-free in $D$, there must exist some arc $(x_i, y') \in A(D')$ with $y' \in Y'$.
  However, by the definition of $Y'$, we have that $(y', x_{i+1}) \in A(D)$ which contradicts the fact that $P$ is zigzag-free in $D$.

  We now claim that $D'$ is $(k - 1)$-tight.
  Let $S' \subset X'$ with $|S'| = k - 1$.
  We need to show that there exists $y' \in Y'$ such that $y'$ is adjacent to every $x \in S'$.
  Let $S = S' \cup \{x_{i + 1}\}$.
  Since $D$ is $k$-tight, there exists $y'\in Y$ such that $y'$ is adjacent to every $x \in S$.
  By the definition of $Y'$, it follows that $y' \in Y'$.
  Therefore, $D'$ is $(k - 1)$-tight.

  By the  induction hypothesis applied to $D'$ and $P'$, there exist paths $P'_1$ and $P'_2$ in $D'$ which satisfy  conditions (i) through (iv).
  Without loss of generality, assume that $ter(P'_1) = x_i$ and $ter(P'_2) = y'$, for some $y' \in Y'$.
  Let $P_1 = P'_1 \concat (x_i, y)$ and $P_2 = P'_2 \concat (y', x_{i + 1}) \concat x_{i + 1}P$.
  We claim that $P_1$ and $P_2$ meet conditions (i) through (iv).
  Conditions (iii) and (iv) obviously hold.
  Condition (i) holds because $P_1'$ and $P_2'$ are disjoint by induction hypothesis and neither vertex $y$ nor any vertex of $x_{i + 1}P$ are vertices of $D'$.
  Condition (ii) holds because $|P_1'| + |P_2'| = i + k$ by induction hypothesis.
  Therefore
  \[|P_1| + |P_2| = |P_1'| + |P_2'| + |X| - i + 1 = |X| + k + 1\]
  and the proof is complete.
\end{proof}

\begin{theorem}
  \label{the:linial_split}
  Let $D[X,Y]$ be a spine digraph. Then, $\pi_k(D) \leq \alpha_k(D)$.
\end{theorem}
\begin{proof}

  We may assume that $D$ is $k$-tight, otherwise the result follows by Theorem~\ref{the:k_loose}.
  We may also assume that $\lambda(D) \leq |X|$, otherwise the result follows by Lemma~\ref{lem:path_cardinaliti_x_plus_one}.
  Let $P = (x_1, x_2, \ldots, x_\ell)$ be a Hamiltonian path in $D[X]$.
  Clearly $P$ is zigzag-free in $D$.
  By Lemma~\ref{lem:fish-bone-paths}, there exists disjoint paths $P_1$ and $P_2$ in $D'$ such that $|P_1| + |P_2| = |X| + k + 1$.
  Note that $|P_i| > k$, for $i = \Set{1, 2}$, otherwise $P_{3 - i}$ would be larger than $|X|$.
  Let $\sP = \Set{P_1, P_2} \cup \Set{(y) : y \notin V(P_1) \cup V(P_2)}$.
  It is easy to see that $\sP$ is a path partition in $D$.
  The $k$-norm of $\sP$ is $|\sP|_k = \Min{|P_1|, k} + \Min{|P_2|, k} + |Y| - k - 1 =   |Y| + k - 1$.
  So, $\pi_k(D) \leq |Y| + k - 1$.
  By Lemma~\ref{lem:lower_bound_linial}, we know that $\alpha_k(D) \geq |Y| + \Min{|X|, k - 1} = |Y| + k - 1$ and the result follows.
\end{proof}

\begin{corollary}
  If $D$ is a split digraph, then $\pi_k(D) \leq \alpha_k(D)$.
\end{corollary}

\section{Acknowledgments}

The first and last authors were supported by National Counsel of Technological and Scientific Development of Brazil, CNPq (grants: 141216/2016-6, 311373/2015-1 and 477692/2012-5).

\bibliographystyle{elsarticle-num}
\bibliography{references.bib}

\end{document}